\documentclass[a4paper,11pt]{article}

\usepackage{graphicx}
\usepackage{amsthm}
\usepackage{amsmath}
\usepackage{amsfonts,amssymb}
\usepackage{dsfont}
\usepackage{psfrag}

\newcommand \HH{{\tilde{h}}}
\newcommand \setR{\mathbb{R}}
\newcommand \OO{\Omega}
\newcommand \virg{ \, , \,\,}
\newcommand \ONE{\mathds{1}}
\newcommand \norme[1]{\left\|#1\right\|}
\newcommand \abs[1]{\left|#1\right|}
\newcommand \tr{\gamma_0}

\newtheorem{prop}{Proposition}{\bf}{\it}
\newtheorem{lemma}{Lemma}{\bf}{\it}

\begin{document}

\title{Approximation of single layer distributions by Dirac masses in Finite Element computations}

\author{B. Fabr\`eges\footnote{Universit\'e Paris-Sud 11, Laboratoire de Math\'ematiques, B\^atiment 425, 91405 Orsay. E-mail : benoit.fabreges@math.u-psud.fr}        \and
        B. Maury\footnote{Universit\'e Paris-Sud 11, Laboratoire de Math\'ematiques, B\^atiment 425, 91405 Orsay. E-mail : bertrand.maury@math.u-psud.fr}
}

\maketitle

\begin{abstract}
We are interested in the finite element solution of elliptic problems with a right-hand side of the single layer distribution type. Such problems arise when one aims at accounting for a physical hypersurface (or line, for bi-dimensional problem), but also in the context of fictitious domain methods, when one aims at accounting for the presence of an inclusion in a domain (in that case the support of the distribution is the boundary of the inclusion). The most popular way to handle numerically the single layer distribution in the finite element context is to spread it out by a regularization technique. An alternative approach consists in approximating the single layer distribution by a combination of Dirac masses.  As the Dirac mass in the right hand side does not make sense at the continuous level, this approach raises particular issues. 
The object of the present paper is to give a theoretical background to this  approach. We present a rigorous numerical analysis of this approximation, and  we present two examples of application of the main result of this paper. The first one is a Poisson problem with a single layer distribution as a right-hand side and the second one is another Poisson problem where the single layer distribution is the Lagrange multiplier used to enforce a Dirichlet boundary condition on the boundary of an inclusion in the domain. Theoretical analysis is supplemented by numerical experiments in the last section.
\end{abstract}

\section{Introduction}
\label{intro}
We are interested in the numerical handling of elliptic problems with a right-hand side of the single layer distribution type, i.e. problems of the form

\begin{eqnarray*}
-\Delta u & = & \varphi \delta_\gamma \quad \text{ in } \OO, \\
u & = & 0 \quad \text{ on } \partial \OO
\end{eqnarray*}
where $\OO$ is a bounded domain in  $\setR^d$,
and  $\varphi \delta_\gamma$,  in $H^{-1}(\Omega)$, can be formally written as 
\begin{equation*}
\left < \varphi \delta_\gamma, v \right >_{H^{-1}, \, H^1_0}  = \int_\gamma \varphi v 
\end{equation*}
where  $\gamma$ is  a $d-1$ dimensional smooth manifold (typically the boundary of a connected subdomain $\omega\subset\subset \OO$) and $\varphi$ is in a Sobolev space $H^{-1/2 + s}(\gamma)$ for $s > 0$.

This type of problem is commonly encountered in numerical modeling. Peskin and Mc Queen~\cite{Peskin} modeled the motion of the heart wall by a collection of elastic fiber immersed in a fluid. The force applied to the fluid by these fibers is a single layer distribution. 
In a similar context, Gerbeau et al.~\cite{Gerbeau}  modeled more recently a valve by an elastic interface immersed in a fluid.
Likewise, Cottet and Maitre in \cite{CottetMaitre} track the interface of an elastic membrane immersed in a fluid by a level set method and the elastic forces on the membrane appears in the right-hand side of the Navier-Stokes equations. In \cite{Herrmann} the interface between two different fluids is tracked with a level set method and the single layer distribution is the surface tension.

This kind of problem also appears in the context of fictitious domain methods. 
For examples in \cite{Glowinski} and \cite{Bertrand} the authors use Lagrange multipliers to enforce Dirichlet boundary conditions. These Lagrange multipliers act as singular distributions of forces supported by the boundary. In~\cite{Pan}, such an approach is proposed to simulate the sedimentation of rigid particles in a fluid. The rigid body constraint is enforced with Lagrange multipliers supported on the boundary of the rigid bodies.
The single layer distribution also appears in the so called Fat Boundary Method (see \cite{Maury01}, or \cite{Bertoluzza} for a full analysis of the method). The resolution of a Poisson like problem in a domain with holes with this method consists of splitting the Poisson problem into two new problems: a global problem, which is solved on a regular mesh, and a local problem around the holes, which introduces the single layer distribution on the boundary of the holes.

There exist several  approaches to solve these problems numerically, the choice of one of them depending on the general framework of the discrete problem. In \cite{Peskin}, the authors have chosen to solve the problem in a finite difference framework. They take a lattice on the whole domain to solve the fluid part of the problem and a uniform collection of points on the fibers, which are not necessarily the same as the points of the lattice. In order to take into account the force of the fibers to the fluid, they regularize the Dirac Delta function on every points on the fibers to spread the single layer distribution out on the computational lattice. This is the so called Immersed Boundary (IB) method. Applications of the IB method can be found in section 9 of \cite{Peskin2002} and, for instance, in \cite{IBGivelbergBunn,IBMittalIaccarino,IBTysonJordanHebert,IBKimPeskin} for more recent ones. Another way to deal with the Dirac Delta functions is to directly inject the jump of the normal derivatives of the corresponding solution in the finite difference scheme. This method, called the Immersed Interface Method (IIM), has been introduced by Leveque and Li in \cite{IIMLevequeLi}. For applications of the IIM, one can look in \cite{IIMTanLeLiLimKhoo} and the references therein.

In a variational context, one can compute the integrals involving the single layer distribution of the variational formulation in some ways. It is the case in a finite element context as in \cite{Glowinski} where the authors use a constant piece-wise element to discretize the Lagrange multiplier space and then compute the integrals exactly. The IB method has been adapted to the finite element framework where one can deal with the Dirac Delta function in a variational way (see \cite{FEMIBBoffiGastaldi,FEMIBIlincaHetu}). In the same spirit, the idea is to discretize the single layer distribution by using a combination of Dirac masses at a collection of points describing the interface. That is the method used in \cite{Gerbeau,Pan,Bertrand} to approximate the Lagrange multiplier defined on an interface of the computational domain. This method requires no additional mesh but only the collection of points on the interface and is convenient to implement in a code. Moreover the integrals which appear in the variational formulation are fast and easy to compute. 

Up to our knowledge, the latter approach has not been justified from a theoretical point of view. 
One of the difficulties is that a well defined linear functional, the initial right-hand side which we will consider in $H^{-1}$, is replaced by a combination of Dirac masses, which do {\em not} make sense at the continuous level (as soon as the dimension $d$ is greater than $1$).

The aim of this paper is to give a numerical analysis of this method, in the case of a scalar  Poisson problem with a single layer distribution as a right hand side in the two dimensional setting. This paper is structured as follow: In section \ref{sec:1} we  establish an error estimate when $\varphi$ is in $H^{-1/2 + s}$ with $0 < s \leq 1/2$ and show that there is a saturation of the order once $s$ is greater than $0.5$ (when $\varphi$ is at least in $L^2$). Then, in section \ref{sec:2} we present some examples on which this result can be used to estimate the error. In section \ref{sec:3}, we present the numerical results to validate the error analysis done in section \ref{sec:1}.


\section{General theorems}
\label{sec:1}
Let $\Omega$ be a domain  in $\mathbb{R}^2$ and let $\omega$ be a smooth subdomain with boundary $\gamma$. We consider a single layer distribution supported by $\gamma$:
\begin{equation*}
\varphi \delta_\gamma \in H^{-1}(\Omega)
 \; : \; v \in V = H^1_0(\OO) \longmapsto   \left< \varphi ,\, v \right>_{H^{-1/2 +s}(\gamma), \, H^{1/2 - s}(\gamma)}, 
\end{equation*}
with $\varphi \in H^{-1/2 + s}(\gamma)$, $0\leq s < 1$.
The expression above makes sense as $v\in H^1_0(\Omega)$, so that its trace on $\gamma$ is in $H^{1/2}(\gamma)\subset H^{1/2-s}$.

Given a conforming  triangulation $T_h$ of $\OO$, we denote by $V_h$ the associated $P^1$ finite element space:
\begin{equation*}
V_h = \left \{v_h \in C^0(\overline\OO) \virg  {v_h }_{| K } \hbox{ is affine }\virg \forall K \in T_h
\right \}.
\end{equation*}

 We are interested in approximating $\varphi \delta_\gamma$ over $V_h$ by a combination  of Dirac masses. 
 Note that it does not make sense at the continuous level, as $V' = H^{-1}(\OO)$ contains no Dirac mass. 
Approximation  properties of such an appropriate combination will be expressed by Proposition~\ref{prop:phiphih}. It is mainly based on the following two lemmas.

\begin{lemma}
\label{lemma:Hs}
Let $I$ be the unit interval $(0,1)$, $h>0$, and $(T_h)$ a family of quasi-uniform triangulations of $I$. We shall represent $T_h$ by its subintervals 
$\gamma_1$, \dots, $\gamma_N$ (we drop the explicit dependence of $\gamma_i$ upon $h$).  Quasi-uniformity  expresses
\begin{equation*}
ch \leq \left |\gamma_i \right | \leq  Ch\virg \hbox{ with } 0<c<C.
\end{equation*}
We now consider $x_1$, \dots, $x_N$, with $x_i\in \gamma_i$.
For any $v \in H^1(I)$, we denote by $v_h$  the corresponding piecewise constant interpolant (see Fig.~\ref{fig:vh})
\begin{equation*}
v_h = \sum_{i=1}^N v(x_i) \ONE_{\gamma_i}.
\end{equation*}
For any $r$, $0 \leq r < 1/2$, one has
\begin{equation*}
\| v-v_h \|_{H^{r}} \leq C h^{1-r} \abs {v}_1
\end{equation*}
where $\abs {v}_1$ is the $H^1$ seminorm, and the fractional derivative Sobolev norm is defined by
\begin{equation*}
\| w \|_{H^{r}}  = \left(\int _I \abs{w}^2 \right)^{1/2} +\left (  \int_I \int_I \frac {\abs {w(y)-w(x)}^2}{\abs{y-x}^{1+2r}}\right ) ^{1/2},
\quad 0 < r < 1/2.
\end{equation*}

\end{lemma}

\begin{figure}
\psfrag{x1}{$x_1$}
\psfrag{x2}{$x_2$}
\psfrag{xN}{$x_N$}
\centering
\includegraphics{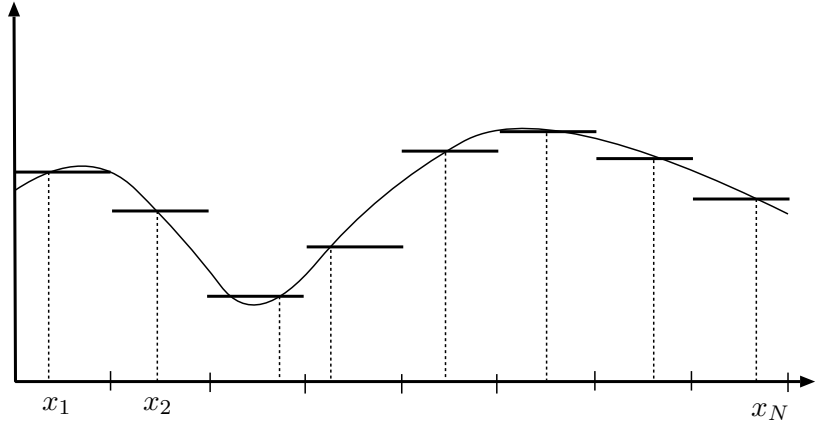}
\caption{$P^0$ interpolant}
\label{fig:vh}
\end{figure}

\begin{proof}
Let us start with the integral over $I\times I$. Setting $w_h = v-v_h$, one has
\begin{eqnarray*}
 \int_I \int_I \frac {\abs {w_h(y)-w_h(x)}^2}{\abs{y-x}^{1+2r}} &=&
 \sum_i \sum_j \int_{\gamma_i} \int_{\gamma_j}\frac {\abs {v(y)-v(x_j)-v(x)+v(x_i)}^2}{\abs{y-x}^{1+2r}}  \\
& =&  \sum_i  \int_{\gamma_i} \int_{\gamma_i}\frac {\abs {v(y)-v(x)}^2}{\abs{y-x}^{1+2r}} +
 \sum_i \sum_{j\neq i}  \int_{\gamma_i} \int_{\gamma_j} \frac {\abs {w_h(y)-w_h(x)}^2}{\abs{y-x}^{1+2r}}  \\
 &=& A + B.
\end{eqnarray*}
The first term (diagonal terms) writes
\begin{eqnarray*}
A = \sum_i  \int_{\gamma_i} \int_{\gamma_i}\frac {\abs {v(y)-v(x)}^2}{\abs{y-x}^{1+2r}}  & \leq & 
\sum_i  \int_{\gamma_i} \int_{\gamma_i}\frac {\int_{\gamma_i} \abs{v'(t)}^2\, dt}{\abs{y-x}^{2r}}  \\
&\leq &\frac 1 {(1-2r)} h^{2-2r}
\int_{I} \abs{v'(t)}^2.
\end{eqnarray*}
As for extradiagonal terms, let us first note that, on any $\gamma_i$, 
\begin{equation*}
\abs {w_h } ^2 = \abs {\int_{x_i}^x v'(t)\, dt } ^2  
\leq  h  \int_{\gamma_i }  \abs{v'(t)}^2.
\end{equation*}
The second term $B$ can then be estimated
\begin{eqnarray*}
B &  \leq& 4 \sum_i \sum_{j\neq i}  \int_{\gamma_i} \int_{\gamma_j} \frac {\abs {w_h(x)}^2}{\abs{y-x}^{1+2r}}  \\
&\leq & 4 h \sum_i   \int_{\gamma_i }  \abs{v'(t)}^2 \sum_{j\neq i}  \int_{\gamma_i} \int_{\gamma_j} \frac {1}{\abs{y-x}^{1+2r}} .
\end{eqnarray*}
In the case $r>0$, the last sum  above (over $j\neq i$) is less than twice
\begin{eqnarray*}
\int_{0}^h dx \int _h^{1} \frac {dy}{\abs{y-x}^{1+2r}}
& = & -\frac 1 {2r} \int_{0}^h dx  \left [(y-x)^{-2r}
\right ]_h^1
\leq 
\frac 1 {2r} \int_{0}^h   (h-x)^{-2r}\, dx \\
& \leq & \frac 1{2r(1-2r)} h^{1-2r},
\end{eqnarray*}
so that
\begin{equation*}
B \leq  \frac 4 {r(1-2r)}  h^{2-2r}  \int_{I}  \abs{v'(t)}^2.
\end{equation*}
Thus, one obtains
\begin{equation*}
 \int_I \int_I \frac {\abs {w_h(y)-w_h(x)}^2}{\abs{y-x}^{1+2r}}  \leq
 \frac C {r (1-2r)} h^{2-2r} \abs { v}^2_{1,I} .
\end{equation*}
The $0$-th order term is simply
\begin{equation*}
\int_I \abs { w_h}^2  = \sum_i \int_{\gamma_i} \abs { w_h}^2 \leq h^2 \abs { v}^2_{1,I} .
\end{equation*}
Hence the order $1-r$ of the lemma.
\qed
\end{proof}

\paragraph{Remarks}
\begin{itemize}
\item If $r = 0$, the $H^r$ norm is only the $L^2$ norm and the result comes from the $0$-th order term of the proof.

\item This result may be ruled out in the case $-1/2 < r < 0$. Indeed, even though the semi-norm part of the definition of the $H^r$ norm can be extended for $r<0$, it is not the case for the $L^2$ norm. One cannot get rid of this $L^2$ norm either, because the constant functions must be controlled in a way. Moreover, the integral over $I \times I$ is of order $h^{2-2r}$ even if $r < 0$, so it must also be true for the expression that replaces the $L^2$ norm and controls the constant functions. For instance, taking the average of the function is not sufficient because it is still a term of order $1$. That is why it may be that the order cannot go further than $1$, even for $r < 0$.
\end{itemize}
The  second lemma is a trace theorem for discrete functions.

\begin{lemma}
\label{lemma:trace}
Let $\OO$ be the unit square in  $\setR^2$, and $\omega\subset\subset \OO$ a smooth subdomain, with boundary  $\gamma$. Let $(T_h)_h$ be a regular  family of triangulations, and $V_h$ the associated $P^1$ Finite Element space. The trace operator $\tr$ maps $V_h$  (equipped with $H^1$ norm) onto $H^1(\gamma)$, with 
\begin{equation*}
\abs {\tr (v_h)}_{1,\gamma} \leq\frac  C {h^{1/2}} \abs {v_h}_{1,\OO} \, .
\end{equation*}

\end{lemma}

\begin{proof}
Because $v_h$ is a $P^1$ function, its gradient is constant in every mesh cell of $T_h$. Let $Q_\gamma$ be the set of all the mesh cells which have a non empty intersection with $\gamma$, one has
\begin{eqnarray*}
\abs {\tr (v_h)}^2_{1,\gamma} = \int_\gamma \abs{\nabla v_h}^2 & = & \sum_{Q \in Q_\gamma} \abs{ \nabla v_h}^2_{\infty, Q} \abs{Q \cap \gamma} \\
& = & \frac{1}{h^2} \sum_{Q \in Q_\gamma} \abs{v_h}^2_{1, Q} \abs{Q \cap \gamma} \, .
\end{eqnarray*}

Assuming that the radius of curvature of $\gamma$ is greater than $h$, the following inequality holds,
\begin{equation*}
\abs{Q \cap \gamma} \leq C h \, ,
\end{equation*}
where $C$ is a constant independent of h. Thus, the previous equality writes,
\begin{eqnarray*}
\abs {\tr (v_h)}^2_{1,\gamma} & \leq & \frac{C}{h} \sum_{Q \in Q_\gamma} \abs{v_h}^2_{1, Q} \\
& \leq & \frac{C}{h} \abs{v_h}^2_{1, \OO} \, ,
\end{eqnarray*}
which is the estimation of the lemma.
\qed
\end{proof}
The next proposition is the approximation error of a single layer distribution by a combination of Dirac masses.

\begin{prop}
\label{prop:phiphih}
Let $\OO$ be the unit square in  $\setR^2$, and $\omega\subset\subset \OO$ a smooth subdomain, with boundary  $\gamma$. Let $(T_h)_h$ be a regular  family of triangulations, and $V_h$ the associated $P^1$ Finite Element space. Let $0 < s \leq 1/2$ and $\varphi \in H^{-1/2+s}(\gamma)$. 

Let $(S_\HH)$ be a family of quasi-uniform triangulations of $\gamma$. We shall represent $S_\HH$ by its subintervals  $\gamma_1$, \dots, $\gamma_{N_\HH}$.

We now consider $x_1$, \dots, $x_{N_\HH}$, with $x_i\in \gamma_i$ and the following approximation of $\varphi$,
\begin{equation}
\label{def:varphih}
\varphi_h^\HH = \sum_{i=1}^{N_\HH} \lambda_i \delta_{x_i}
\end{equation}
where the real numbers $\lambda_i$ are chosen as follows,
\begin{equation}
\label{def:lambdai}
\lambda_i = \left< \varphi , \mathds{1}_{\gamma_i} \right> \,.
\end{equation}
There exists a constant $C > 0$ such that :
\begin{equation}
\label{ErrorEstimate}
\left| \left< \varphi , v_h \right> -  \left< \varphi_h^\HH , v_h \right> \right| \leq C ~ \sqrt{\frac{\HH}{h}}  ~ \HH^{s} ~  \| \varphi \|_{-1/2+s,\gamma} | v_h |_{1,\Omega}
\end{equation}
for all function $v_h$ in $V_h$ and where $\left< ~ . ~ , ~ . ~ \right>$ denotes the dual pairing between $H^{-1/2+s}(\gamma)$ and $H^{1/2-s}(\gamma)$. 
\end{prop}

\begin{proof}

Thanks to the definition of $\varphi_h^\HH$, one has, 
\begin{eqnarray*}
\left| \left< \varphi , v_h \right> -  \left< \varphi_h^\HH , v_h \right> \right|  &=& \left| \left< \varphi , v_h \right> - \left< \varphi , \sum_{i=1}^{N_\HH} v_h(x_i) \mathds{1}_{\gamma_i} \right> \right| \\
& \leq & \| \varphi \|_{-1/2+s,\gamma} \|v_h - \sum_{i=1}^{N_\HH} v_h(x_i) \mathds{1}_{\gamma_i} \|_{1/2-s,\gamma} \, .
\end{eqnarray*}
The function $\tilde{v}_\HH$ defined by 
\begin{equation*}
\tilde{v}_\HH = \sum_{i=1}^{N_\HH} v_h(x_i) \mathds{1}_{\gamma_i}
\end{equation*}
is the piecewise constant interpolant of $v_h$ on $\gamma$. Now, in order to use lemma \ref{lemma:Hs} we perform a change of variable to map the curve $\gamma$ onto $\left[0, 1\right]$,
\begin{equation*}
\begin{array}{lccc}
X :  & \left[0, 1\right] & \longrightarrow & \gamma \\
& t & \longrightarrow & X(t) \, .
\end{array}
\end{equation*}
The variable $t$ here is the arc length so $X(t)$ is the point on $\gamma$ such that the length of the arc $X(0)X(t)$ is equal to $\abs{\gamma} t$. Thus, denoting by $w^\HH_h$ the function $v_h - \tilde{v}_\HH$, one has,
\begin{eqnarray*}
\abs{w^\HH_h}^2_{1/2 - s, \gamma} & = & \int_\gamma \int_\gamma \frac{\abs{w^\HH_h(y) - w^\HH_h(x)}^2}{\abs{y - x}^{2 - 2s}} \\
& = & \int_0^1 \int_0^1 \frac{\abs{w^\HH_h\circ X (t) - w^\HH_h \circ X (\tau)}^2}{\abs{X(t) - X(\tau)}^{2 - 2s}} X'(t) X'(\tau) dt d\tau \, .
\end{eqnarray*}
For any $t$ in $[0, 1]$, $X'(t)$ is equal to $\abs{\gamma}$ so the change of variables writes,
\begin{equation*}
\abs{w^\HH_h}_{1/2 - s, \gamma} = \abs{\gamma} \abs{w^\HH_h \circ X}_{1/2 - s, [0, 1]} \, .
\end{equation*}
The same goes for the $L^2$ norm. Now, the length of the subintervals $\gamma_i$ of $\gamma$ is $\HH$, so the length of the corresponding subintervals in $[0, 1]$ is $\HH/\abs{\gamma}$.
We apply the lemma \ref{lemma:Hs} with $r = 1/2 - s$ and $h = \HH / \abs{\gamma}$,
\begin{eqnarray*}
\norme{v_h - \tilde{v}_\HH}_{1/2 - s, \gamma} & = & \abs{\gamma} \|v_h \circ X - \tilde{v}_\HH \circ X \|_{1/2-s, [0, 1]} \\
& \leq & \frac{\abs{\gamma} C}{\abs{\gamma}^{1/2 + s}}~ \HH^{1/2 + s} \abs{v_h \circ X}_{1, [0, 1]} \, .
\end{eqnarray*}
Again, with the change of variables, one has,
\begin{equation*}
\norme{v_h - \tilde{v}_\HH}_{1/2 - s, \gamma} \leq \frac{C}{\abs{\gamma}^{1/2 + s}} ~ \HH^{1/2 + s} \abs{v_h}_{1, \gamma} \, .
\end{equation*}

Now, the trace lemma \ref{lemma:trace} allows us to extend the $H^1$ semi-norm on $\gamma$ of $v_h$ over the whole domain $\OO$,
\begin{equation*}
\|v_h - \tilde{v}_\HH \|_{1/2-s,\gamma} \leq \frac{C}{\abs{\gamma}^{1/2 + s}} ~ \frac{\HH^{1/2 + s}}{h^{1/2}} \abs{v_h}_{1, \OO} \, ,
\end{equation*}
which ends the proof.
\qed
\end{proof}

\paragraph{Remark 1} :
If $1/2 < s \leq 1$, one cannot expect to have a better order than the one with $s=1/2$. Indeed, here is a one dimensional example where $\varphi$ is a smooth function and the error is still of order $1$. 

The domain is $(0, \, 1)$ and $\varphi$ is the constant function  $1$. We take $h = \HH/2$ and the function $v_h$ is defined by 
\begin{equation*}
v_h(x) = 
\left\{
\begin{array}{ll}
\vspace{0.2cm}
-\displaystyle{\frac{\abs{x-\HH/2}}{h}} + 1 & \text{ in } [0, \, \HH] \\
0 & \text{ elsewhere.}
\end{array}
\right.
\end{equation*}
If we take the point $\HH/2$ to approximate $\varphi$ in $[0,\,  \HH]$, the error writes,
\begin{equation*}
\abs{\left< \varphi, \, v_h \right> - \left< \varphi_h^\HH, \, v_h \right>} = \abs{\frac{\HH}{2} - \left(\int_0^\HH \varphi\right) v_h\left(\frac{\HH}{2}\right)} = \frac{\HH}{2} \, .
\end{equation*}
Thus the error is of order one and there is a saturation of the order once $s$ is greater than $0.5$.

\paragraph{Remark 2} : In the case $s=0$ the definition of the approximation of the single layer distribution is different because the characteristic functions are not in $H^{1/2}$. Nevertheless, these characteristic functions can be replaced by a partition of unity and the result of Proposition \ref{prop:phiphih} holds. Thus, there is still convergence if $\HH$ is such that $\HH/h$ goes to $0$ as $h$ goes to $0$.

\section{Applications}
\label{sec:2}
\subsection{Error estimate for the Poisson problem}
\label{subsec:21}
We consider here a Poisson problem with homogeneous Dirichlet boundary conditions and a single layer distribution as a right-hand side. We are interested in the error estimate when the single layer distribution is approximated by a combination of Dirac masses.

\begin{prop}
\label{prop:poisson}
Let $\OO$ be the unit square and $\omega \subset\subset \OO$ a smooth subdomain, with boundary  $\gamma$. We consider the following Poisson problem,
\begin{equation*}
\left|
\begin{array}{rcll}
- \Delta u & = & \varphi \delta_\gamma & \mbox{ in } \OO\vspace{3mm} \\
u & = & 0 & \mbox{ on } \partial \OO
\end{array}
\right.
\end{equation*}
where $\varphi$ is in $H^{-1/2 + s}(\gamma)$, $0 < s \leq 1/2$.

Let $(T_h)_h$ be a regular family of triangulations and $V_h$ the associated $P^1$ Finite Element space. We approximate $\varphi$ by a function $\varphi_\HH^h$ according to Eqs.~(\ref{def:varphih}) and~(\ref{def:lambdai}). Denoting by $u_h$ the finite element solution of $u$ and assuming that $\HH$ is of the  order of  $h$,  the error estimate is
\begin{equation*}
\| u - u_h \|_1  \leq C \left( \sqrt{h} | u |_2  + h^s \| \varphi \|_{-1/2 + s, \gamma} \right) \,,
\end{equation*}
where $|.|_2$ denotes the $H^2$ semi-norm.
\end{prop}

\begin{proof}
Thanks to Strang's lemma (see \cite{CiarletEF}), the error estimate is,
\begin{equation*}
\| u - u_h \|  \leq C  \left( \inf_{v_h \in V_h} \| u-v_h \| +  \sup_{w_h \in V_h} \frac{| \left< \varphi, w_h \right> - \left< \varphi^\HH_h , w_h \right> |}{\| w_h \|} \right) \, .
\end{equation*}

The first part is the usual error and is of order $1/2$ in our case because $T_h$ is a non conformal mesh and $V_h$ is the Finite Element space of $P^1$ function.
The second part is given by Proposition \ref{prop:phiphih}. Thus, the error estimate writes,
\begin{equation*}
\| u - u_h \|_1  \leq C \left( \sqrt{h} | u |_2  + \sqrt{\frac{\HH}{h}} ~ \HH^s \| \varphi \|_{-1/2 + s, \gamma} \right) \, .
\end{equation*}
In particular, if $\HH$ is equal to $h$, the error estimate becomes,
\begin{equation*}
\| u - u_h \|_1  \leq C \left( \sqrt{h}| u |_2  + h^s \| \varphi \|_{-1/2 + s, \gamma} \right) \, ,
\end{equation*}
which ends the proof.
\qed
\end{proof}

\subsection{Error estimate for the approximation of the solution of a saddle-point problem}
\label{subsec:22}
We consider a Poisson equation in a domain with a hole and homogeneous Dirichlet boundary conditions on each boundary. The homogeneous condition on the boundary of the hole is enforced by Lagrange multipliers defined on the boundary. These Lagrange multipliers are approximated by a combination of Dirac masses. First we prove an abstract theorem on the error approximation when the discrete Lagrange multipliers space is not included in the continuous Lagrange multipliers space. Second, we use this theorem and Proposition \ref{prop:phiphih} to give an error estimate of the Poisson problem.

\begin{prop}
\label{prop:abstract_theo}
Let V be a Hilbert space, $a$ be a bounded elliptic bilinear form in $V$ and $f$ be in $V'$. Let $\Lambda$ be another Hilbert space, $B\in \mathcal{L}(V,\Lambda)$ and $K$ the kernel of $B$. We consider the problem of finding $u$ in  $K$ such that $a(u, v) = \left< f, \, v \right>$ for all $v$ in $K$, and its sadle-point formulation
\begin{equation*}
\begin{array}{lcll}
a(u, v) + \left< B^* \lambda, \, v \right> & = & \left< f ,\, v \right> & \forall v \in V \\
\left(\mu, B u \right) & = & 0 & \forall \mu \in \Lambda \, .
\end{array}
\end{equation*}
Let $V_h$ be a finite dimensional subspace of $V$ and $\Lambda_\HH$ be a finite dimensional space, not necessarily included in $\Lambda$. Let $B_\HH^h$ be a bounded linear application from $V$ to $\Lambda_\HH$, we denote by $K_\HH^h$ the approximation of $K$
\begin{equation*}
K_\HH^h = \left\{v_h \in V_h, ~ \left( (B_\HH^h)^* v_h, \, \mu_\HH \right) = 0 ~~ \forall \mu_\HH \in \Lambda_\HH \right\} \, .
\end{equation*}
The approximate saddle-point problem is
\begin{equation*}
\begin{array}{lcll}
a(u_h, v_h) + \left< (B_\HH^h)^* \lambda_\HH, \, v_h \right> & = & \left< f ,\, v_h \right> & \forall v_h \in V_h \\
\left(\mu_\HH, B_\HH^h \, u_h \right) & = & 0 & \forall \mu_\HH \in \Lambda_\HH \, .
\end{array}
\end{equation*}
We have the following error estimate
\begin{equation*}
\| u - u_h \|_V \leq \left( 1 + \frac{\| a \|}{\alpha} \right) \inf_{w_h \in K_\HH^h} \| w_h - u \| + \frac{1}{\alpha} \inf_{\mu_\HH \in \Lambda_\HH}{\|\xi - \left(B^h_\HH\right)^* \mu_\HH \|}_{V'}
\end{equation*}
where $\alpha$ is the coercivity constant of a and  $\xi$ is the linear form in $V$ defined by 
\begin{equation*}
a(u,v) + \left< \xi , v \right> = \left< f , v \right> \quad \forall v \in V \, .
\end{equation*}

\end{prop}

\begin{proof}
$u_h$ is the solution of the approximate saddle-point problem, so we have,
\begin{equation*}
a(u_h,v_h) = \left< f , v_h \right> \quad \forall v_h \in K_\HH^h \, .
\end{equation*}
For all $w_h \in K_\HH^h$, we write $v_h = u_h - w_h \in K_\HH^h$. Thus,
\begin{eqnarray*}
a(v_h , v_h ) & = & a(u_h - w_h , v_h ) \\
& = & \left< f , v_h \right> - a(w_h, v_h) \, .
\end{eqnarray*}
As
$u$ is the solution of our problem we have,
\begin{equation*}
a(u, v_h) + \left< \xi , v_h \right> = \left< f , v_h \right> \, .
\end{equation*}
Thus,  for any $\mu_\HH$ in $\Lambda_\HH$, one has,
\begin{eqnarray*}
a(v_h,v_h) & = & a(u,v_h) + \left< \xi , v_h \right> - a(w_h,v_h) - ( B^h_\HH v_h , \mu_\HH )\\
& = & a(u - w_h , v_h) + \left< \xi - \left(B^h_\HH\right)^* \mu_\HH , v_h \right> \, .
\end{eqnarray*}
Taking the absolute value of the previous equality gives us
\begin{equation*}
\alpha \| v_h \|^2 \leq \| a \| \| u - w_h \| \|v_h \| + \| \xi - \left(B^h_\HH\right)^* \mu_\HH\|_{V'} \|v_h \|  \, .
\end{equation*}
But $v_h = u_h - w_h$ so we have
\begin{equation*}
\alpha \| u_h - w_h \| \leq \| a \| \|u - w_h \| + \| \xi - \left(B^h_\HH\right)^* \mu_\HH \|_{V'}  \, .
\end{equation*}
Thus,
\begin{eqnarray*}
\|u-u_h \| & \leq & \|u - w_h \| + \|w_h - u_h \| \\
& \leq & \left(1 + \frac{\| a \|}{\alpha} \right) \|w_h - u\| + \frac{1}{\alpha} \| \xi - \left(B^h_\HH\right)^* \mu_\HH\|_{V'} \quad \forall w_h \in K_\HH^h~, \mu_\HH \in \Lambda_\HH \,,
\end{eqnarray*}
which concludes the proof.
\qed
\end{proof}

Now we define the problem we are interested in. Let $\Omega$ be a bounded domain in $\mathbb{R}^2$ and $\mathcal{O}$ a non empty open subset of $\Omega$. We denote by $\mathcal{P}$ the following problem
\begin{equation*}
\left(\mathcal{P}\right) \left|
\begin{array}{rcll}
- \Delta u & = & f & \mbox{ in } \Omega \setminus \mathcal{O} \\
u & = & 0 & \mbox{ on } \gamma \\ 
u & = & 0 & \mbox{ on } \partial \Omega \, ,
\end{array}
\right.
\end{equation*}
where $f$ is a function of $L^2(\Omega \setminus \mathcal{O})$ and $\gamma$ denotes the boundary of $\mathcal{O}$. We use a fictitious domain method and extend the function $f$ by $0$ in $\mathcal{O}$. We still denote by $f$ this extension and we denote by $V$ the space $H^1_0(\Omega)$ and by $K$ the constrained space,
\begin{equation*} 
K = \left\{ v \in V ; ~ v_{|_\gamma} = 0 \right\} \, .
\end{equation*}

The Lagrange multipliers space is $\Lambda = L^2_0(\gamma)$, which is the space of $L^2$ functions on $\gamma$ with zero mean value. Let $B$ be the following application,
\begin{equation*}
\begin{array}{lccc}
B : & V & \longrightarrow & \Lambda \\
& v &  \longrightarrow & v_{|\gamma} \, .
\end{array}
\end{equation*}
We have $K = \ker B$ and the corresponding saddle-point problem is:

Find $(u,\lambda) \in V \times \Lambda$, such that
\begin{equation*}
\begin{array}{rcll}
\vspace{0.2cm}
\displaystyle{\int_\Omega \nabla u \cdot \nabla v + \int_\gamma \lambda v} & = & \displaystyle{\int_\Omega f v} & \quad \forall v \in V \\
\displaystyle{\int_\gamma \mu u} & = & 0 & \quad \forall \mu \in \Lambda \, .
\end{array}
\end{equation*}
Let $\HH >0$, we discretize the boundary $\gamma$ by taking $N_\HH$ points $x_{i,\HH} \in \gamma$ so that the distance between two consecutive points is $\HH$. Let $V_h \subset V$ be a finite element space approximation of $V$. The Lagrange multipliers space and the constrained space are approximated as follows:
\begin{equation*} 
\Lambda_\HH = \mathbb{R}^{N_\HH} 
\end{equation*}
\begin{equation*}
K^h_\HH =  \left\{ v_h \in V_h ; ~v_h(x_{i, \HH}) = 0 \right\} \, .
\end{equation*}
Here the space $\Lambda_\HH$ is not included in $\Lambda$. The application $B$ becomes
\begin{equation*}
\begin{array}{lccc}
B_\HH^h : & V_h & \longrightarrow & \Lambda_\HH \\
& v_h &  \longrightarrow & \left(v_h(x_{i,\HH})\right) _{i = 1 \ldots N_\HH} \, .
\end{array}
\end{equation*}
We denote by $(\mathcal{P}_{h,\HH})$ the following approximated problem:

Find $(u_h,\lambda_\HH) \in V_h \times \Lambda_\HH$, such that
\begin{equation*}
\begin{array}{lcll}
\vspace{0.1cm}
a(u_h, v_h) + \left< \lambda_H , \, B_\HH^h v_h \right>  & = & \left< f , \, v_h \right> & \forall v_h \in V_h \\
\left< \mu_\HH , \, B_\HH^h u_h \right> & = & 0 & \forall \mu_\HH \in \Lambda_\HH \, .
\end{array}
\end{equation*}
Where $a$ is the bilinear form defined in $V \times V$ by
\begin{equation*}
a(u,v) = \int_\Omega \nabla u \cdot \nabla v \, .
\end{equation*}

We can apply the result of Proposition \ref{prop:abstract_theo} so we have the following error estimate,
\begin{equation*}
\| u - u_h \|_V \leq \left( 1 + \frac{\| a \|}{\alpha} \right) \inf_{w_h \in K_\HH^h} \| w_h - u \| + \frac{1}{\alpha} \inf_{\mu_\HH \in \Lambda_\HH}{\|\xi - \left(B^h_\HH\right)^* \mu_\HH \|}_{V'} \, .
\end{equation*}
As in the previous example, the first part of this error estimate is the usual one and the second part is the error done by approximating the linear form $\xi$ with a combination of Dirac masses. Indeed, for any $\mu_\HH = (\mu_i)_{1, \ldots, N_\HH}$ in $\Lambda_\HH$ and for any $v_h$ in $V_h$, one has,
\begin{equation*}
\left< \left(B^h_\HH\right)^* \mu_\HH ~ , ~ v_h \right> = \left(\mu_\HH ~ , ~ B^h_\HH v_h \right) 
=  \sum_{i=1}^{N_\HH} \mu_i v_h(x_i) 
= \left< \sum_{i=1}^{N_\HH} \mu_i \delta_{x_i} ~ , ~ v_h \right> \, .
\end{equation*}
The error estimate of Proposition \ref{prop:abstract_theo} becomes,
\begin{equation*}
\| u - u_h \|_1 \leq \left( 1 + \frac{\| a \|}{\alpha} \right) \inf_{w_h \in K_H^h} \| w_h - u \| + \frac{1}{\alpha} \inf_{\mu_\HH \in \Lambda_\HH}{\frac{\left< \xi - \sum_{i=1}^{N_\HH} \mu_i \delta_{x_i} ~ , ~ v_h \right>}{\|v_h\|}} \, .
\end{equation*}
Thus, by applying the result of Proposition \ref{prop:phiphih},
\begin{equation*}
\| u - u_h \|_1  \leq C \left( \sqrt{h} \| u \|_1  + \sqrt{\frac{\HH}{h}} ~ \HH^{1/2} \| \xi \|_{0, \gamma} \right) \, .
\end{equation*}
Again, if $\HH$ is equal to $h$, the error estimate writes,
\begin{equation*}
\| u - u_h \|_1  \leq C \sqrt{h} \left( \| u \|_1  + \| \xi \|_{0, \gamma} \right) \, .
\end{equation*}

\section{Numerical results}
\label{sec:3}
In this section we present two examples to validate the estimate~(\ref{ErrorEstimate}) of Proposition \ref{prop:phiphih}. The first one is a one dimensional example to show that in the case $\HH = h$ the error order is the one predicted by Proposition \ref{prop:phiphih}. The second example is a two dimensional numerical example where the solution of a Poisson equation with a right-hand side of the single layer distribution type is computed. This right-hand side is the Laplacian of a function with a jump of its normal derivatives across an interface. The exact solution is thus known and we compute the error in the case $\HH = h$ for various values of $s$.

\subsection{$1$d example}
Let $\varphi$ be the function defined by
\begin{equation*}
\varphi = \frac{1}{x^{1-s}} \, ,
\end{equation*}
where $s$ is positive, and $v_h$ be the affine function defined in $[0, \, \HH]$ by
\begin{equation*}
v_h(x) =  -\frac{\abs{x - \frac{\HH}{2}}}{h} + 1 \, .
\end{equation*}
The estimate of Proposition \ref{prop:phiphih} can be computed explicitly. Indeed, one has,
\begin{equation*}
\left< \varphi, \, v_hÊ\right> = \int_0^\HH \varphi v_h = \frac{\HH^s}{s} + \left(\frac{1}{s(s+1)} - \frac{1}{s(s+1)2^s} - \frac{1}{2s} \right) \frac{\HH^{s+1}}{h} \,.
\end{equation*}
Let $x_i$ be in $[0, \, \HH]$, the second part is
\begin{equation*}
\left< \varphi_h^\HH, \, v_hÊ\right> = \left(\int_0^\HH \varphi \right) v_h(x_i) = \left(-\frac{\abs{x_i - \frac{\HH}{2}}}{h} + 1\right) \frac{\HH^s}{s} \,.
\end{equation*}
Let $0 \leq \alpha \leq 1$, we write $x_i = \alpha \HH$. Thus, the error estimate writes
\begin{equation*}
\abs{\left< \varphi, \, v_hÊ\right> - \left< \varphi_h^\HH, \, v_hÊ\right>} = C(s) \frac{\HH^{s+1}}{h} \, , 
\end{equation*}
where $C(s)$ is a constant depending only on $s$. Now, if $h = C \HH$, where $C$ is a constant, the error is of order $s$ which is the order predicted by Proposition \ref{prop:phiphih}.

\subsection{Two-dimensional  example}
In this section $\Omega$ is the unit square and $\gamma$ is a circle of radius $R$. To investigate the behavior of the approach for a right-hand side $\varphi \delta_\gamma$, with $\varphi$ in $H^{-1/2 + s}(\gamma)$, we build the function $\varphi$ as an infinite series of sine functions
\begin{equation*}
\varphi = C \sum_{n=1}^\infty n^{-s} \sin(n \theta)
\end{equation*}
We now build a function $u$ such that the jump of its normal derivatives on $\gamma$ is $\varphi$ so that $u$ is the solution of a Poisson problem of the form
\begin{equation*}
\left|
\begin{array}{rcll}
-\Delta u & = & \varphi \delta_\gamma + f & \text{ in } \Omega \\
u & = & 0 & \text{ on } \partial \Omega
\end{array}
\right.
\end{equation*}
where $f$ is to be defined later on.

First we need a cut function to make sure that the solution $u$ is equal to zero on the boundary of $\Omega$. Let us denote by $\chi_\epsilon$ the $C^2$ function defined in $\mathbb{R}$ by
\begin{equation*}
\chi_\epsilon(x) = \left\{
\begin{array}{ll}
\vspace{0.2cm}
1 & \text{ if } x \leq 0 \\
\vspace{0.2cm}
\displaystyle{\frac{-6}{\epsilon^5} x^5 + \frac{15}{\epsilon^4} x^4 - \frac{10}{\epsilon^3} x^3 + 1} & \text{ if }Ê0 < x < \epsilon \\
0 & \text{ if } x \geq \epsilon
\end{array}
\right.
\end{equation*}
We define $\rho$ as the distance between $\gamma$ and $\partial \Omega$ (see figure \ref{fig:distance}). We denote by $R_{max}$ the sum of $R$ and $\rho$.
\begin{figure}
\psfrag{m}{$\rho$}
\psfrag{g}{$\gamma$}
\psfrag{O}{$\Omega$}
\psfrag{R}{$R$}
\psfrag{Rmax}{$R_{max}$}
\psfrag{chi1}{$\chi = 1$}
\psfrag{chi0}{$\chi = 0$}
\centering
\includegraphics[width=0.6 \linewidth]{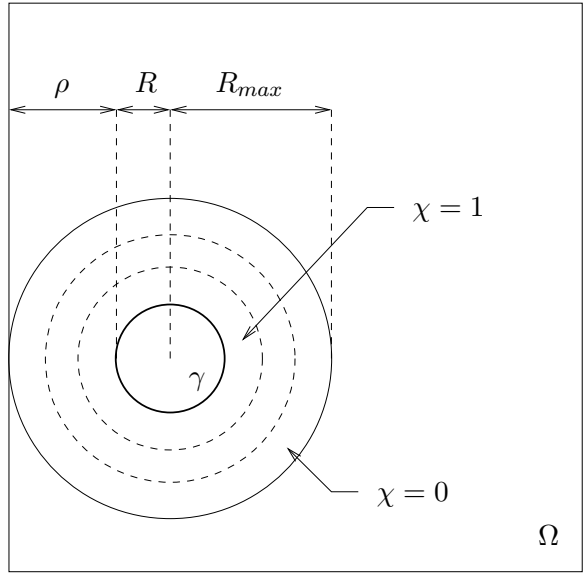}
\caption{Notations}
\label{fig:distance}
\end{figure}
Now, we can define the following function :
\begin{equation*}
u(r, \theta) = \left\{
\begin{array}{ll}
\vspace{0.3cm}
\displaystyle{\chi_{\rho/3}\left(r - \left(R+\frac{\rho}{3}\right)\right) \sum_{n = 1}^{\infty} \left(\frac{r}{R}\right)^{-n} n^{-s-1} \sin(n \theta)} & \quad \mathrm{ ifÊ}  ~ r \geq R \\
\displaystyle{u \left(r+\left(1-\frac{r}{R}\right)R_{max}\right)} & \quad \mathrm{ if } ~ r<R. \\
\end{array}
\right.
\end{equation*}
It is a function which has a jump of its normal derivative on $\gamma$ and oscillates very fast close to $\gamma$. Its Laplacian outside the particle is given by
\begin{eqnarray*}
\Delta u (r, \theta) & = &\chi_{\rho/3}^{''}\left(r - \left(R+\frac{\rho}{3}\right)\right) \sum_{n = 1}^{\infty}\left(\frac{r}{R}\right)^{-n} n^{-s-1} \sin(n \theta) \\
& + & \chi_{\rho/3}^{'}\left(r - \left(R+\frac{\rho}{3}\right)\right) \sum_{n = 1}^{\infty}\frac{1-2n}{R} \left(\frac{r}{R}\right)^{-n-1} n^{-s-1} \sin(n \theta) \, .
\end{eqnarray*}
The jump $\varphi$ of its normal derivatives is :
\begin{eqnarray*}
\varphi & = & \left( \left( 1 - \frac{R_{max}}{R} \right) - 1\right) \frac{\partial u}{\partial r}(R, \theta) \\
& = & \frac{R_{max}}{R^2} \sum_{n = 1}^{\infty} n^{-s} \sin(n \theta)
\end{eqnarray*}
which is in $H^{-1/2 + s - \epsilon}(\gamma)$ for all $\epsilon > 0$. This means that we have the exact solution of the following problem
\begin{equation*}
\left|
\begin{array}{rcll}
-\Delta u & = & \varphi \delta_\gamma + f & \text{ in } \Omega \\
u & = & 0 & \text{ on } \partial \Omega
\end{array}
\right.
\end{equation*}
where $f$ is given by the Laplacian of $u$ outside the particle.

To compute the numerical solution of this problem we truncate the series at some order $N$ greater than $2 \pi / h$ so that there is at least one period of the function $\sin(N \theta)$ in a mesh cell. We compute the numerical solution for several $s$ using $Q_1$ finite elements and compare it to the exact solution (see figure \ref{fig:ErrH1}). In all the tests, the order $N$ is equal to $2^{12}$ and $h$ takes its values from $2^{-6}$ to $2^{-10}$. The space step $\HH$ is equal to $h$ so that the $H^1$ error should be of the form,
\begin{equation*}
\| u - u_h \|_1  \leq C \left( \sqrt{h} |u|_2  + h^s \| \varphi \|_{-1/2 + s, \gamma} \right).
\end{equation*}

We can see in figure \ref{fig:ErrH1} that the error order seems to stay close to $0.6$ once $s$ is greater than $0.5$. This is because of the space approximation of order $1/2$ since we use a cartesian mesh with finite elements of order $1$, but it could also be because of the saturation of the error order when $s$ is greater than $0.5$ (see the remark right after Proposition \ref{prop:phiphih}).

Figure \ref{fig:OrderS} shows the order of the error as $s$ goes to $1$. When $s$ is less than $0.5$, the slope is equal to $1$ which is expected. The numerical order is always better than the theoretical one, but this kind of super-convergence often happens when one compute numerical solutions of PDE.

\begin{figure}
\centering
\includegraphics[width = 0.90 \linewidth]{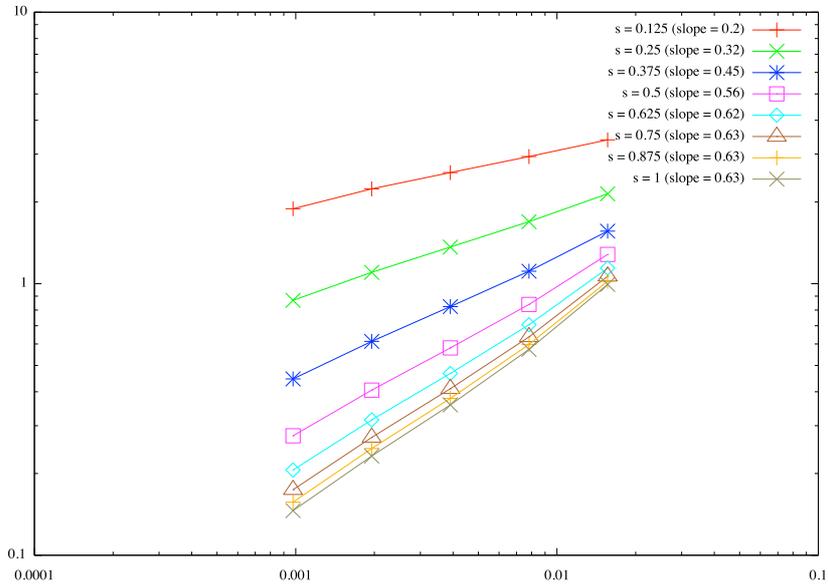}
\caption{$H^1$ error for some value of $s$}
\label{fig:ErrH1}
\end{figure}

\begin{figure}
\centering
\includegraphics[width = 0.90 \linewidth]{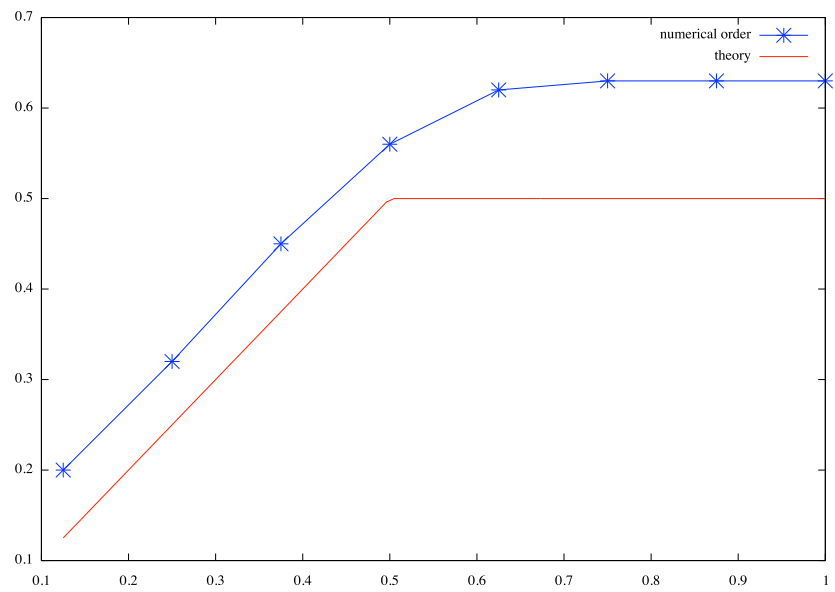}
\caption{Order of the $H^1$ error with respect to $s$}
\label{fig:OrderS}
\end{figure}

\end{document}